\documentclass[11pt]{article}

\usepackage{geometry}                
\usepackage{array}

\usepackage{amssymb}

\usepackage{amsmath}
\usepackage{amsthm, amssymb, amsfonts}

\usepackage[pdftex]{graphicx}

\usepackage{tikz}
\usetikzlibrary{calc,through,backgrounds}
\usetikzlibrary{arrows,decorations.pathmorphing,backgrounds,fit}
\usetikzlibrary{shapes.geometric} 

\newtheorem{thm}{Theorem}[section]
\newtheorem{cor}[thm]{Corollary}
\newtheorem{lem}[thm]{Lemma}
\newtheorem{prop}[thm]{Proposition}
\newtheorem{eg}[thm]{Example}
\theoremstyle{remark}
\newtheorem{rem}[thm]{Remark}

\theoremstyle{definition}

\usepackage{hyperref}

\newcommand{\field}[1]{\mathbb{#1}}

\newcommand{\inv}{^{-1}}

\newcommand{\sym}[1]{\mathrm{Sym}(#1)}
\newcommand{\alt}[1]{\mathrm{Alt}(#1)}

\newcommand{\dih}[1]{\mathrm{Dih}(#1)}
\newcommand{\frtw}{\mathrm{Frob}(20)}
\newcommand{\cyc}[1]{{\mathrm{C}_{#1}}}
\newcommand{\zent}[1]{\mathrm{Z}(#1)}

\newcommand{\aut}[1]{\mathrm{Aut}( #1)}		
\newcommand{\syl}[2]{\mathrm{Syl}_{#1}( #2)}	
\newcommand{\vst}[2]{G_{#1}^{[ #2]}}		
\newcommand{\vstx}[1]{G_x^{[ #1 ]}}			
\newcommand{\vsty}[1]{G_y^{[ #1 ]}}			

\newcommand{\oP}[2]{\mathrm{O}^{#1}(#2)}	
\newcommand{\nml}{\vartriangleleft}					
\newcommand{\norm}[2]{\text{N}_{#1}(#2)}				
\newcommand{\cent}[2]{\text{C}_{#1}(#2)}				
\newcommand{\grp}[1]{\langle #1 \rangle}				
\newcommand{\core}[2]{\textbf{core}_{#1}(#2)}			
\newcommand{\ga}{G_e}					

\newcommand{\gx}{G_x}
\newcommand{\gxy}{G_{xy}}

\newcommand{\gy}{G_y}

\newcommand{\GL}[2]{\mathrm{GL}_{#1}(#2)}		
\newcommand{\SL}[2]{\mathrm{SL}_{#1}(#2)}					
\newcommand{\PSL}[2]{\mathrm{PSL}_{#1}(#2)}					

\newcommand{\gam}{\Gamma}			
\newcommand{\del}{\Delta}				

\newcounter{claim}[thm]

\title{On vertex stabilisers in symmetric quintic graphs}

\author{G. L. Morgan}

\newcommand{\foo}[1]{%
\begin{tikzpicture}[#1]%
\colorlet{sincolor}{white}%
\draw[thick] (0,0) -- (0,1);%
\draw[thick] (-0.75,-0.75) -- (0,0);%
\draw[thick] (0.75,-0.75) -- (0,0);%
\draw[style=sincolor] (-1,0)--(-1.2,0);%
\draw[style=sincolor] (1,0)--(1.2,0);%
\draw[style=sincolor] (0,1)--(0,1.2);%
\end{tikzpicture}%
}

\begin{document}

\maketitle

\begin{abstract}
In this paper we determine all locally finite and symmetric actions of a group on the tree of valency five. As a corollary we complete the classification of  the isomorphism types of vertex and edge stabilisers in a group acting symmetrically on a  graph of valency five. This builds on work of Weiss and recent work of Zhou and Feng. This  depends upon the  second result of this paper, the classification of isomorphism types of finite, primitive amalgams of degree $(5,2)$.
\end{abstract}

\section{Introduction}

\label{sec:intro}

Let $\gam$ be a connected graph and let $G$ be a subgroup of $\aut{\gam}$, the group of permutations of the vertices of $\gam$ which preserve the edges of $\gam$.  The action of $G$ is said to be \emph{locally finite} if for each vertex $x$ of $\gam$ the stabiliser in $G$ of $x$, denoted $G_x$, is a finite group (which implies the edge stabilisers are  also finite). When $G$ acts transitively on 
the set of ordered pairs of adjacent vertices we  say that the action is \emph{symmetric}, or that $\gam$ is $G$-\emph{symmetric}. If the valency of the (necessarily regular) graph is odd, this is equivalent to transitivity on both vertices and  edges. The first result of this paper is a theorem about such actions on the tree of valency five, for which we write $\gam_5$.

\begin{thm}
\label{thm:bigthm2}
Let  $G\leq \aut{\gam_5}$ be such that the action of $G$ is  both  locally finite and symmetric. Then $G$ has  one of the presentations given in Tables \ref{table:pres} and \ref{table:pres2}. Conversely, each of the presentations given in Tables  \ref{table:pres} and \ref{table:pres2} defines a subgroup of $\aut{\gam_5}$ with a locally finite and symmetric action.
\end{thm}

The above theorem depends upon the classification of finite, primitive amalgams of degree (5,2). Certain cases have been completed by Weiss \cite{Weiss-pres} and the results of Zhou and Feng in \cite{zhoufeng} are a contribution. Here we complete the classification.

\begin{thm}
\label{thm:bigthm}
There are  exactly 25 isomorphism classes of finite, primitive amalgams of degree $(5,2)$ and they are uniquely determined by their type. The  types of the amalgams are exactly those listed in  Table \ref{table:bigtable} and presentations for their universal completions are given in Tables \ref{table:pres} and \ref{table:pres2}.
\end{thm}

\begin{table}
\begin{center}
\begin{tabular}{ c| c| c| c |c}
\hline	\hline
Amalgam &	$A_1$	& 	$A_2$	&	$B$  & s \\ \hline	\hline
$\mathcal{Q}_1^1$ & 	$\cyc{5}$	&	$\cyc{2}$	&	1 & 1 \\	\hline
$\mathcal{Q}_1^2$ & 	$\dih{10}$	&	$2^2$	&	$\cyc{2}$ & 1 \\	\hline
$\mathcal{Q}_1^3$ &  $\dih{10}$	&	$\cyc{4}$	&	$\cyc{2}$ & 1 \\	\hline
$\mathcal{Q}_1^4$ &	$\dih{20}$	&	$\dih{8}$	&	$2^2$ & 1\\	\hline \hline
$\mathcal{Q}_2^1$ & 	$\frtw$	&	$\cyc{4}\times\cyc{2}$	&	$\cyc{4}$ & 2 \\ 	\hline 
$\mathcal{Q}_2^2$ &  	$\frtw$	&	$\cyc{8}$	&	$\cyc{4}$ & 2 \\ 	\hline
$\mathcal{Q}_2^3$ &  	$\frtw$	&	$\dih{8}$	&	$\cyc{4}$ & 2\\ 	\hline
$\mathcal{Q}_2^4$ &  	$\frtw$	&	$\mathrm{Q}_8$	&	$\cyc{4}$ & 2  \\ \hline 
$\mathcal{Q}_2^5$ &	$\frtw \times \cyc{2}$	&	$\mathrm{N}_{16}$	&	$\cyc{4}\times \cyc{2}$ & 2\\ 	\hline
$\mathcal{Q}_2^6$ & 	$\frtw \times \cyc{2}$	&	$\mathrm{M}_{16}$	&	$\cyc{4}\times \cyc{2}$ & 2\\ 	\hline
$\mathcal{Q}_2^7$ &  	$\alt{5}$	&	$\sym{4}$	&	$\alt{4}$ & 2 \\ 	\hline
$\mathcal{Q}_2^8$ &  	$\alt{5}$	&	$\alt{4}\times \cyc{2}$	&	$\alt{4}$ & 2 \\ 	\hline
$\mathcal{Q}_2^9$ &  	$\sym{5}$	&	$\sym{4}\times \cyc{2}$	&	$\sym{4}$ & 2 \\ \hline \hline
$\mathcal{Q}_3^1$ & 	$\frtw \times \cyc{4}$	&	$\cyc{4}\wr\cyc{2}$	&	$\cyc{4}\times \cyc{4}$ & 3 \\ \hline
$\mathcal{Q}_3^2$ &  	$\alt{5} \times \alt{4}$	&	$\alt{4}\wr \cyc{2}$	&	$\alt{4}\times\alt{4}$ & 3 \\ 	\hline
$\mathcal{Q}_3^3$ & 	$\sym{5}\foo{scale=.15}\sym{4}$	&	$\mathrm{L}_1$	&	$\sym{4}\foo{scale=.15}\sym{4}$ & 3\\ 	\hline
$\mathcal{Q}_3^4$ &	 	$\sym{5}\foo{scale=.15}\sym{4}$	&	$\mathrm{L}_2$	&	$\sym{4}\foo{scale=.15}\sym{4}$ & 3 \\ 	\hline
$\mathcal{Q}_3^5$ & 	$\sym{5}\times\sym{4}$	&	$\sym{4}\wr \cyc{2}$	&	$\sym{4}\times\sym{4}$ & 3 \\ \hline \hline

$\mathcal{Q}_4^1$ &	$2^4:\alt{5}$	&	$2^{2+2+2}:\sym{3}$	&	$2^4:\alt{4}$ & 4 \\	\hline
$\mathcal{Q}_4^2$ &	$2^4:\alt{5}$	&	$2^{2+2+2}: \cyc{6}$	&	$2^4:\alt{4}$ & 4 \\	\hline
$\mathcal{Q}_4^3$ &	$2^4:(\alt{5}\times \cyc{3})$	&	$(2^{2+4}:3) : \sym{3}$	&	$2^{2+4}:3^2$ & 4 \\	\hline
$\mathcal{Q}_4^4$ &	$2^4:(\alt{5}\times \cyc{3})$	&	$(2^{2+4}:3) : \cyc{6}$	&	$2^{2+4}:3^2$ & 4 \\ 	\hline
$\mathcal{Q}_4^5$ & 	$2^4:\sym{5}$	&	$2^{2+4+1}:\sym{3}$	&	$2^{2+4}:\sym{3}$ & 4 \\ 	\hline
$\mathcal{Q}_4^6$ & 	$2^4 : \sym{5} \foo{scale=.15} \sym{3}$	&	$2^{2+4}:\sym{3}^2$	&	$2^{2+4}:\sym{4}\foo{scale=.15} \sym{3}$ & 4\\ 	\hline \hline
$\mathcal{Q}_5^1$ &  	$2^6: \sym{5}\foo{scale=.15} \sym{3}$	&	$(2^6 : (\alt{4} \times \cyc{3})) : \cyc{4}$	&	$2^6 : \sym{4}\foo{scale=.15} \sym{3}$ & 5 

\end{tabular}
\caption{The finite, primitive amalgams of degree (5,2)}
\label{table:bigtable}
\end{center}
\end{table}

Amalgams turn up in various areas of group theory, often when only information about part of a group is known. For example we may have a group $G$ and subgroups $H$ and $K$ such that the triple $(H,K,H\cap K)$ holds some intrinsic property of the group $G$. To capture this idea, an amalgam is defined to be a 5-tuple $\mathcal{A}=(A_1,A_2,B,\pi_1,\pi_2)$ of three groups $A_1$, $A_2$, $B$ and two monomorphisms $\pi_1$, $\pi_2$ such that $\pi_i : B \rightarrow A_i$ for $i=1,2$ (note there is no mention of an ambient group containing $A_1$ and $A_2$).   We call the pair of indices $(|A_1 : \pi_1(B)|,|A_2:\pi_2(B)|)$ the \emph{degree} of the amalgam.
 If $B$ is a finite group and the degree of the amalgam is a pair of integers, we say the amalgam is finite. We are interested in \emph{primitive} amalgams, these are amalgams in which the only subgroup of $B$  that  is normal in both $A_1$ and $A_2$ is the trivial subgroup. 

The connection between Theorems \ref{thm:bigthm2} and \ref{thm:bigthm} is delivered by completions of amalgams and  coset graphs. For details and notation see Section \ref{sec:pre}. Making use of coverings of symmetric quintic graphs  by the quintic tree we have the following.

\begin{cor}
\label{bigthm1}
Suppose that $\gam$ is a connected $G$-symmetric quintic graph such that the action of $G$ is locally finite. Let $e=\{x,y\}$ be an edge of $\gam$. Then  $(G_x,G_e)=(A_1,A_2)$ for one of the rows of  Table \ref{table:bigtable}, $G$ is a quotient of a universal completion of a finite, primitive  amalgam  of degree (5,2) and $\gam\cong \gam_G(\mathcal{A})$, the coset graph corresponding to the completion $G$ of the amalgam $\mathcal{A}$.
\end{cor}

 Using Theorem \ref{thm:bigthm} we are able to measure the transitivity of each  quintic symmetric graph in terms of arc-transitivity. An $s$-arc ($s\in\field{N}$) in a graph $\gam$ is an ordered sequence $(x_0,x_1,\dots,x_s)$ of $s+1$ vertices such that $x_{i-2}\neq x_{i}$ for $i\in[2,s]$.  We  say that $\gam$ is $(G,s)$-transitive where $s$ is the largest integer such that the group $G$ acts on $\gam$ and acts transitively on the set of $s$-arcs of $\gam$. Thus a $G$-symmetric graph is $(G,s)$-transitive for some $s\geq 1$. We also say that $\gam$ is locally $(G,s)$-transitive where $s$ is the largest integer such that for each vertex $x\in \gam$ the group  $G_x$ acts transitively on the set of $s$-arcs with initial vertex $x$. We make use of a subgroup which is only locally transitive in Section \ref{sec:valofs}. The following corollary gives information on the arc-transitivity of symmetric quintic graphs.
 
 \begin{cor}
 \label{cor:stranscor}
 Let $\gam$ be a $(G,s)$-transitive quintic graph for some $s\geq 1$. Then $s\in [1,5]$ and $G$ is a completion of an amalgam isomorphic to $\mathcal{Q}_s^j$ appearing in Table \ref{table:bigtable}. \end{cor}

Symmetric graphs of valency three were first studied by Tutte \cite{tutte} who bounded the order of a vertex stabiliser. Using this result, Djokovi\'{c} and Miller \cite{djokmiller} classified the finite,primitive amalgams of degree $(3,2)$. There are exactly seven types of amalgam which fall into five families for $s\in [1,5]$. In particular for $s=2$ and $s=4$ there are two types of amalgams which differ in the following way. In an amalgam of the ``first kind" the group $\ga$ splits over $\gxy$, whereas in an amalgam of the ``second kind" this does not occur.

In \cite{Gold1} it is shown that there are fifteen possible isomorphism types of finite, primitive amalgams of degree $(3,3)$, the  Goldschmidt Amalgams, which arise from semisymmetric trivalent graphs.  Symmetric graphs of valency four show a different pattern, in \cite{djokovic} Djokovic gives infinitely many amalgams of degree $(4,2)$ which are pair-wise non-isomorphic. This result is part of a general phenomena, in \cite{djokbolyai} it is shown that there are infinitely many amalgams of degree $(k,2)$ whenever $k$ is composite, and in the edge-transitive case, infinitely many amalgams of degree $(m,n)$ whenever one of $m$ or $n$ is composite, see \cite[7.13]{basskulkarni}. We point the reader to  \cite{sami1} and \cite{sami2}  to see that there is some hope in the composite case however. It is conjectured in \cite{djokbolyai} that there are finitely many finite, primitive amalgams of degree $(p,q)$ where $p$ and $q$ are primes. Our result not only confirms this for $p=5$, $q=2$, but also enumerates the possible structures.

Returning to the case of symmetric quintic graphs, Weiss in \cite{Weiss-pres} considered the cases of highly transitive action. Recently, Zhou and Feng  found the isomorphism types of the vertex stabiliser \cite{zhoufeng}, under the assumption this group is soluble. It was shown that there are six isomorphism types of vertex stabiliser here, and that $s\in [1,3]$. The isomorphism type of the edge stabiliser and the isomorphism types of the  amalgams were not determined however, we consider this problem in Section \ref{sec:gxy1=1sol}. In the non-soluble case we find 7 primitive amalgams when $\vst{xy}{1}$ (to be defined) is trivial. When $\vst{xy}{1}\neq 1$ we show that $s\geq 4$ and  use \cite{Weiss-pres} to determine the amalgam.

 The amalgams of Theorem \ref{thm:bigthm} fall into 5 families depending on the value of $s$. In contrast to  \cite{djokmiller} we find that there is a unique family which contains a single member and there are many types of amalgam where $\ga$ does not split over $\gxy$. For the definitions of $\mathrm{L}_1$ and $\mathrm{L_2}$ see Lemma \ref{lem:defnofl1l2}, the definitions of $\mathrm{N}_{16}$ and $\mathrm{M}_{16}$ appear before Lemma \ref{lem:edge16} and we define a group isomorphic to $\sym{m}\foo{scale=.15}\sym{n}$ at the end of this introduction.

By \cite[Theorem 1, pg.209]{djokmiller} every finite symmetric quintic graph can be obtained from a graph such as $\gam_G(\mathcal{A})$ where $\mathcal{A}$ is a primitive amalgam of degree (5,2) and $G=\gx *_{\gxy}\ga$, the free amalgamated product of $G_x$ and $\ga$ over $G_{xy}$. Although the theorem is stated  for cubic graphs, it readily generalises to other regular graphs. Specifically, given a normal subgroup $N$ of finite index with $N\cap (\gx \cup \ga)=1$ we pass to the (finite) group $G/N$ and we see this is also a completion of $\mathcal{A}$. Then $\gam_{G/N}(\mathcal{A})$ is a finite graph.  This method is exploited in \cite{conderlorimer} to give answers to some of the open problems from \cite{djokmiller} and in \cite{condernedela} to give the full list of symmetric cubic graphs with at most 768 vertices. To do this a presentation for the group $\gx *_{\gxy}\ga$ is needed, thus Theorem \ref{thm:bigthm2} is the first step towards such a result for  symmetric quintic graphs.

A related area of interest is the study of distance transitive graphs. These are graphs $\gam$ such that whenever $x,y,u,v\in \gam$ are such that $\mathrm{d}(x,y)=\mathrm{d}(u,v)$ there is $g\in\aut{\gam}$ with $x^g=u$ and $y^g=v$. Note that every distance transitive graph is $\aut{\gam}$-symmetric, but the converse is not true in general. The distance transitive graphs of valency five have been classified by Gardiner and Praeger \cite{gardpraeg}. There are precisely fourteen such graphs and  each can be obtained as $\gam_G(\mathcal{A})$ for some amalgam $\mathcal{A}$ appearing in Table \ref{table:bigtable} and some completion $G$ of $\mathcal{A}$.

In the remainder of this introduction we review  notation used when a group acts on a graph and outline the structure of this paper.  All graphs in this paper are connected, without loops and without multiple edges and every action of a group is faithful and locally finite.  Let  $\mathrm{d}(\ ,\ )$ be the usual distance metric on $\gam$. We let $\del^{[i]}(x)=\{y \in \gam \mid \mathrm{d}(x,y) \leq i \}$ and write $\del(x)=\del^{[1]}(x)$, the neighbourhood of $x$.  If a group $G$ acts on $\gam$ and $x$ is a vertex of $\gam$, we define the subgroup 
\[
\begin{split}
\vst{x}{i}= \bigcap_{ u \in \del^{[i]}(x) } G_u.
\end{split}
\] Note that $\vst{x}{i}\nml \vstx{j}$ whenever $i\geq j$. If $(x_0,x_1,\dots,x_s)$ is an $s$-arc of $\gam$ and $i\in \field{N}$ we write $\vst{x_0x_1\dots x_s}{i}$ for $\vst{x_0}{i}\cap \vst{x_1}{i} \cap \dots \cap \vst{x_s}{i}$ (although if $i=0$ we ignore the superscript).  In  a quintic graph for example, if $e=\{x,y\}$ is an edge then the group $\vst{xy}{1}$ is the subgroup of $G$ which fixes every vertex of the following subgraph of $\gam$.

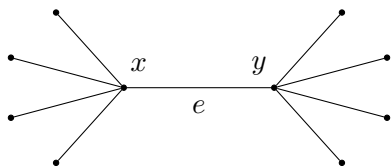
\begin{center}
\begin{tikzpicture}

\node [draw,circle,fill=black,minimum size=2pt,inner sep=0pt] (delta) at (0,0) {};
\node [draw,circle,fill=black,minimum size=2pt,inner sep=0pt] (alpha) at (2,0) {};

\draw [-] (delta) -- (alpha);

\node [draw,circle,fill=black,minimum size=2pt,inner sep=0pt] (tau+1) at (2.9,1) {};
\node [draw,circle,fill=black,minimum size=2pt,inner sep=0pt] (tau+2) at (3.5,0.4) {};
\node [draw,circle,fill=black,minimum size=2pt,inner sep=0pt] (tau+3) at (2.9,-1) {};
\node [draw,circle,fill=black,minimum size=2pt,inner sep=0pt] (tau+4) at (3.5,-0.4) {};
\draw  (tau+1)--(alpha)--(tau+2);
\draw  (tau+3)--(alpha)--(tau+4);

\node [draw,circle,fill=black,minimum size=2pt,inner sep=0pt] (delta+1) at (-0.9,1) {};
\node [draw,circle,fill=black,minimum size=2pt,inner sep=0pt] (delta+2) at (-1.5,0.4) {};
\node [draw,circle,fill=black,minimum size=2pt,inner sep=0pt] (delta+3) at (-0.9,-1) {};
\node [draw,circle,fill=black,minimum size=2pt,inner sep=0pt] (delta+4) at (-1.5,-0.4) {};
\draw  (delta+1)--(delta)--(delta+2);
\draw  (delta+3)--(delta)--(delta+4);

\node (alphalab) at (0.2,0.3) {$x$};
\node (betalab) at (1.8,0.3) {$y$};

\node (gammalab) at (1,-0.25) {$e$};

\end{tikzpicture}
\end{center}
This paper is organised as follows. In Section \ref{sec:pre} we recall some results that we will use in later sections and establish the setup in which we will work for the remainder of the paper. We bound the value of $s$ in Section \ref{sec:valofs} and distinguish between the cases $\vst{xy}{1}=1$ and $s\geq 4$. The isomorphism type of the amalgam in the former case is then determined in Sections \ref{sec:gxy1=1sol} and \ref{sec:gxy1=1nonsol}, whilst in the latter case we appeal to results of Weiss \cite{Weiss-pres}.  Finally in Section \ref{sec:uniandpres} we prove the uniqueness of the primitive amalgams we have found and give presentations for the universal completions of these amalgams. Our notation for groups and group extensions is hopefully self-explanatory. A possible exception   is the notation $\sym{n}  \foo{scale=.15} \sym{m}$ where $n,m\in\field{N}$ are both at least 3. This is the unique index two subgroup of $\sym{n}\times\sym{m}$ which does not have a direct factor isomorphic to $\sym{n}$ or $\sym{m}$. For example, $\sym{3} \foo{scale=.15} \sym{3} \cong \grp{(1,2,3),(4,5,6),(2,3)(4,5)}$.

\section*{Acknowledgement}
This work is part of the author's PhD thesis. The author is sincerely grateful to Prof. C. W. Parker's for guidance and advice. The author is also grateful for support from EPSRC.

\section{Preliminaries}
\label{sec:pre}

Recall from the introduction that an amalgam is a 5-tuple $\mathcal{A}=(A_1,A_2,B,\pi_1,\pi_2)$ of three groups $A_1$, $A_2$, and $B$ and two monomorphisms $\pi_1$, $\pi_2$. We say two amalgams $\mathcal{A}$ and $\mathcal{B}=(C_1,C_2,D,\rho_1,\rho_2)$ are of the same \emph{type} if there are isomorphisms  $\alpha : A_1 \rightarrow C_1$, $\beta : A_2 \rightarrow C_2$, $\gamma :B \rightarrow D$ such that $\mathrm{im}(\alpha\pi_1)=\mathrm{im}(\rho_1\gamma)$ and $\mathrm{im}(\beta\pi_2)=\mathrm{im}(\rho_1\gamma)$. We may then denote the type by the triple $(A_1,A_2,B)$, provided the subgroup $B$ of $A_1$ and $A_2$ to which we refer is clear. Additionally, we say that the two amalgams $\mathcal{A}$ and $\mathcal{B}$ of the same type are \emph{isomorphic} provided the maps $\alpha$, $\beta$ and $\gamma$ can be chosen so that the following diagram commutes.

\begin{center}
\begin{tikzpicture}
\node (a1) at (0,0) {$A_1$};
\node (c) at (2,0) {$B$};
\node (a2) at (4,0) {$A_2$};
\node (b1) at (0,-2) {$C_1$};
\node (b2) at (4,-2) {$C_2$};
\node (d) at (2,-2) {$D$};

\draw [->] (c) to node [above] {$\pi_1$} (a1);
\draw [->] (c) to node [above] {$\pi_2$} (a2);
\draw [->] (c) to node [left] {$\gamma$} (d);
\draw [->] (d) to node [below] {$\rho_1$} (b1);
\draw [->] (d) to node [below] {$\rho_2$} (b2);
\draw [->] (a1) to node [left] {$\alpha$} (b1);
\draw [->] (a2) to node [left] {$\beta$} (b2);

\end{tikzpicture}
\end{center}
 
Non-isomorphic amalgams should have different properties. The following example illustrates this and will turn up again in Lemma \ref{lem:q31 unique}.

\begin{eg}
\label{eg:example1}
Let $A_1 =\grp{a,b,c}\cong \frtw \times \cyc{4}$ where $a$ has order 5, $b$ and $c$ have order 4, $a^b=a^2$ and $ac=ca$. Let $A_2=\grp{d,e,f}\cong \cyc{4} \wr \cyc{2}$ where $d$ and $e$ have order 4, $f$ has order 2, $de=ed$ and $e^f=d$. Let $B=\grp{g,h}\cong \cyc{4}\times\cyc{4}$ where both $g$ and $h$ have order 4 and commute.

For $i=1,2$ we define maps $\pi_i: B \rightarrow A_i$ by giving the images of the generators of $B$. We set $\pi_1(g)=b$ and $\pi_1(h)=c$, $\pi_2(g)=d$ and $\pi_2(h)=e$. Note then that $\core{A_1}{\pi_1(B)}=\grp{c}$ and $\core{A_2}{\pi_2(B)}=\grp{de}$, so that if $K\leq B$ is such that $\pi_i(K)\nml A_i$ for $i=1,2$, then $K=1$. Hence $\mathcal{A}=(A_1,A_2,B,\pi_1,\pi_2)$ is a primitive amalgam and it is of degree (5,2).

We can obtain non-isomorphic amalgams of the same type by adjusting the definition  (but not the image) of $\pi_2$. We now set $\pi_2(g)=d$ and $\pi_2(h)=de$. Then for $K=\grp{h}$ we see that $\pi_i(K) \leq \zent{A_i}$ for $i=1,2$, hence the amalgam obtained with this change to $\pi_2$ is no longer primitive. We  get a slightly different situation when we set $\pi_2(g)=d$ and $\pi_2(h)=de^{-1}$. Then for $K=\grp{h}$ we have $\pi_1(K)\leq \zent{A_1}$, but $\pi_2(K)\nleq \zent{A_2}$. However $\pi_2(K)\nml A_2$, so this third amalgam is also not primitive. All three amalgams have the same type, but are pairwise non-isomorphic.
\end{eg}

The number of isomorphism classes of amalgams of a fixed type is well understood. For $i=1,2$ we denote by $\norm{\aut{A_i}}{\pi_i(B)}$ and $\cent{\aut{A_i}}{\pi_i(B)}$ the subgroups of $\aut{A_i}$ which respectively normalise and centralise $\pi_i(B)$. We define a map $\pi_i^* : \norm{\aut{A_i}}{\pi_i(B)} \rightarrow \aut{B}$ by $\pi_i^* (\alpha) : x \mapsto \pi_i^{-1}\alpha\pi_i(x)$ for $x\in B$. This is a homomorphism with kernel $\cent{\aut{A_i}}{\pi_i(B)}$. For $i=1,2$ we set $A_i^*=\pi_i^*(A_i)$. Using this notation, we state the amalgam counting lemma of Goldschmidt.

\begin{lem}[Goldschmidt's Lemma]
\label{lem:golds lem}
There is a bijection between the isomorphism classes of amalgams of type $(A_1,A_2,B)$ and the $(A_1^*,A_2^*)$-double cosets in $\aut{B}$.
\begin{proof}
See \cite[Lemma 2.7]{Gold1}.
\end{proof}
\end{lem}

In our applications of Goldschmidt's Lemma we will continue to use the notation $A_1^*$ and $A_2^*$, though we may have different names for the groups involved in the amalgam and may deal with multiple amalgams in which the group $B$ features at the same time.

A \emph{faithful completion} of the amalgam $\mathcal{A}$ is a triple $(G,\theta_1,\theta_2)$ where $G$ is a group and $\theta_i$ are  monomorphisms $\theta_i : A_i \rightarrow G$ ($i=1,2$) such that $G=\grp{\theta_1(A_1),\theta_2(A_2)}$ and for all $b\in B$ we have $\theta_1(\pi_1(b))=\theta_2(\pi_2(b))$. A universal completion of $\mathcal{A}$ is a completion $(G,\theta_1,\theta_2)$ such that whenever $(H,\mu_1,\mu_2)$ is also a completion of $\mathcal{A}$, there is a unique map $\kappa : G \rightarrow H$ such that $\kappa\theta_i=\mu_i$ for $i=1,2$. The free amalgamated product of $A_1$ and $A_2$ over $B$, denoted $A_1 *_B A_2$ is a universal completion of $\mathcal{A}$, but finite completions always exist. We sometimes omit the maps from our statements if they are clear from the context or if we have identified $B$ with $\pi_1(B)$ and $\pi_2(B)$.

From a primitive amalgam $\mathcal{A}=(A_1,A_2,B,\pi_1,\pi_2)$ of degree $(k,2)$ and a completion $G$ of $\mathcal{A}$ we construct the graph $\gam_G(\mathcal{A})$ as follows.  We take as vertices the right cosets of $A_1$ in $G$ and say that two cosets $A_1 g$ and  $A_1 h$ are adjacent whenever $gh^{-1}\in A_1 a A_1$ for some $a\in A_2 - B$. One needs to check that $\gam_G(\mathcal{A})$ is well-defined and that the definition does not depend on the choice of $a\in A_2 - B$. (To our knowledge, this construction is due to \cite{miller}). We let $G$  act on $\gam=\gam_G(\mathcal{A})$ by $h : A_1 g \mapsto A_1 gh$. With this action we see $\gam$ is $G$-symmetric and if $A_1\cap A_1^a=B$ then $\gam$ is regular of valency $k$.  The vertex stabilisers in this action are conjugate to $A_1$ and the edge stabilisers to $A_2$. 

Now let $\gam$ be a $G$-symmetric graph (with locally finite action) and pick an edge $e=\{x,y\}$. Let $\pi_x: G_{xy} \rightarrow G_x$ and $\pi_e:G_{xy}\rightarrow G_e$ be the identity embeddings. Set $\mathcal{A}=(\gx$, $\ga$, $\gxy$, $\pi_x$, $\pi_e)$. Then $\mathcal{A}$  is a finite, primitive amalgam. We define $\theta : \gam_G(\mathcal{A}) \rightarrow \gam$   by $\theta:\gx g \mapsto x^g$. We leave the reader to verify that this is an isomorphism of graphs and that $\theta$ commutes with the action of $G$. Thus there is an equivalence between the study of $G$-symmetric graphs of valency $k$ and amalgams of degree $(k,2)$.

From now on, we let $\mathcal{A}=(\gx,\ga,\gxy,\pi_x,\pi_e)$ be a finite, primitive amalgam of index $(p,2)$, $p$ an odd prime, and let $G=\gx *_{\gxy}\ga$ be the universal completion of $\mathcal{A}$. Let $\gam=\gam_G(\mathcal{A})$ and identify $\gx$, $\ga$ and $\gxy$ with their images in $G$. We summarise the relevant properties of $\gam$ below.

\begin{prop}
\label{prop:fund prop of gam}
The following hold.
\begin{itemize}
\item[(i)] The graph $\gam$ is the  $p$-valent tree.
\item[(ii)] If $K\leq \gxy$ and both $\norm{\gx}{K}$ and $\norm{\ga}{K}$ are transitive on $\del(x)$ and $\{x,y\}$ respectively, $K=1$.
\item[(iii)] The graph $\gam$ is $G$-symmetric and the action of $G$ is locally finite.
\item[(iv)] The subgroup of $G$ fixing a vertex, respectively, edge of $\gam$ is  $G$-conjugate to $\gx$, respectively, $\ga$.
\end{itemize}
\begin{proof}
The results are well known. Part (i) is essentially \cite[pg.32]{serre}. Part (ii) follows from primitivity of the amalgam. The first part of  (iii) follows from (ii) and the second part together with (iv) follow from the definition of $\gam$.
\end{proof}
\end{prop}

We call upon part (ii) of the above proposition frequently in our arguments. Since it is obvious in its application, we shall usually suppress reference. 

\begin{lem}
Suppose that $q$ is a prime with $q\mid |\gxy|$. Then $q<p$.
\begin{proof}
Let $q$ be a prime with $q\geq p$ and pick $S\in\syl{q}{\gxy}$. Since $\gxy/\vstx{1}$ is a point stabiliser of $\gx/\vstx{1}$ which is embedded in $\sym{p}$, $q$ cannot divide $|\gxy/\vstx{1}|$ and so $S\leq \vstx{1}$. Hence $S\in\syl{q}{\vstx{1}}$. The Frattini argument now yields $\gx=\norm{\gx}{S}\vstx{1}$ and $\ga=\norm{\ga}{S}\gxy$. In particular, $\norm{\gx}{S}$ and $\norm{\ga}{S}$ are transitive on $\del(x)$ and $\{x,y\}$ respectively,  so $S=1$.
\end{proof}
\end{lem}

Taking $p=3$, the above result implies that $|\gxy|=2^a$ for some $a\in \field{N}$ and that $|\gx|=2^a\cdot3$, which is the situation considered in \cite{djokmiller}. We now leave the general situation and fix $p=5$. We see that $|\gxy|=2^a\cdot3^b\cdot5$ for some $a,b\in\field{N}$. In particular, $\gxy$, $\ga$ and $\vstx{1}$ are soluble groups. The following proposition will be used to make a case division between Sections \ref{sec:gxy1=1sol} and \ref{sec:gxy1=1nonsol}. 

\begin{prop}
\label{prop:solofgx}
The group $\gx$ is soluble if and only if $\gx/\vstx{1}$ is soluble.
\begin{proof}
This is \cite[Proposition 4]{djokbolyai}.
\end{proof}
\end{prop}

Knowing that $\gx$ acts transitively on $\del(x)$ which has order 5 allows us to determine the possible isomorphism type of $\gx/\vstx{1}$. In particular, we can conclude  $\gx/\vstx{1}$ either contains a normal cyclic subgroup of order 5 or a normal subgroup isomorphic to $\alt{5}$.

\begin{lem}
\label{lem:subgpsofs5}
Suppose that $H$ is a transitive subgroup of $\sym{5}$ acting on $5$-points. Then $H$ is isomorphic to one of the following groups: $\cyc{5}$, $\dih{10}$, $\frtw$, $\alt{5}$ or $\sym{5}$.
\begin{proof}
This is an easy calculation in $\sym{5}$.
\end{proof}
\end{lem}

\section{Values of $s$}
\label{sec:valofs}

The graph $\gam$ is the quintic tree by Proposition \ref{prop:fund prop of gam} (i), therefore $\gam$ is bipartite. By $G_0$ we denote the subgroup of $G$ which fixes the parts set-wise. Since $\gam$ is $G$-symmetric, we have $|G:G_0|=2$ and $\gam$ is $G_0$-semisymmetric. Moreover $G_0=\grp{\gx,\gy}$ for the edge $e=\{x,y\}$. The following lemma is surely well known, so we omit the proof.

\begin{lem}
\label{lem:s-trans iff loc s-trans}
Suppose that $G$ acts $s$-transitively  and $G_0$ acts locally $t$-transitively. Then  $s=t$.
\end{lem}

\begin{lem}
\label{lem:gxy1=1}
Suppose that $\vst{xy}{1}\neq 1$ for some edge $\{x,y\}$ of $\gam$. Then $s\geq 4$.
\begin{proof}
We first consider the case where $\gx/\vstx{1}$ is soluble. Lemma \ref{lem:subgpsofs5} shows that $\gx/\vstx{1}$ contains a regular abelian subgroup. Therefore we may apply \cite[Theorem (i)]{Weiss-pfact} which gives $\vst{xy}{1}=1$, a contradiction.

Suppose now that $\gx/\vstx{1}$ is insoluble, and therefore has a normal subgroup isomorphic to $\alt{5}\cong \SL{2}{4}$. We see that $\gx$ acts 2-transitively on $\Delta(x)$ and so $s \geq 2$. By Lemma \ref{lem:s-trans iff loc s-trans}  the group $G_0=\grp{\gx,\gy}$ is locally $s$-transitive, and since $(G_0)_x=G_x$ and $(G_0)_x^{[1]}=\vstx{1}$ we have that $(G_0)_z/(G_0)_z^{[1]}=G_z/\vst{z}{1}$ contains a normal subgroup isomorphic to $\alt{5}$ for each vertex $z$ of $\gam$. Thus we may apply \cite[Theorem 1.1]{WeissBN} which implies $\vstx{1}=1$ if $s=2$ and $\vst{x_0}{1}\cap\vst{x_1}{1}\cap G_{x_3}\cap\dots G_{x_s}=1$ if $s \geq 3$ and $(x_0,\dots,x_s)$ is any $s$-arc. Since $\vst{xy}{1}$ is contained in $G_z$ for any $z \in \del(y)\setminus \{x\}$, we have $s\geq 4$. 
\end{proof}
\end{lem}

\begin{rem}
As observed in \cite[pg.10]{Weiss-pres}, although the results of \cite{WeissBN} are stated for finite graphs, only that the stabilisers of vertices are finite is used in the proof.
\end{rem}

\begin{lem}
\label{lem:gxy1neq1}
Suppose that $\vst{xy}{1}\neq 1$. Then $s\in\{4,5\}$.
\begin{proof}
The previous lemma gives $s\geq 4$. We again make use of the subgroup $G_0=\grp{\gx,\gy}$ which acts locally $s$-transitively on $\gam$. By  \cite[1.2]{WeissBN} we have $s=4$, $5$ or $7$ and we may identify $\gx$ and $\gy$ with the vertex stabilisers of adjacent vertices in the graphs coming from the groups $\mathrm{A}_2(4)$, $\mathrm{B}_2(4)$ and $\mathrm{G}_2(4)$ in the respective cases. However, $\gx$ and $\gy$ are conjugate in $G$, so they are isomorphic, but this property does not hold in the amalgam arising from  $\mathrm{G}_2(4)$. Thus $s\in \{4,5\}$.
\end{proof}
\end{lem}

The case  of quintic symmetric graphs with $s\in \{4,5\}$ is considered in \cite{Weiss-pres} where it is shown that the amalgam $(\gx,\ga,\gxy)$ has a completion in the groups $\aut{\mathrm{PSL}_{3}(4)}$ (for $s=4$) and $\aut{\mathrm{Sp}_4(4)}$ (for $s=5$). The quintic graph can be found as the incidence graph of the point-line geometry of the associated vector space. The result is below.

\begin{thm}[Weiss]
Suppose that $\vst{xy}{1}\neq 1$. Then the amalgam $\mathcal{A}=(\gx,\ga,\gxy)$ is in Table \ref{table:weisstable}.
\begin{proof}
By \ref{lem:gxy1=1} we have $s\in \{4,5\}$. Hence \cite[Theorem 1.2]{Weiss-pres} is applicable. 
\end{proof}
\end{thm}

\begin{table}
\begin{center}
\begin{tabular}{ c| c| c| c |c}
\hline	\hline
Amalgam &	$\gx$	& 	$\ga$	&	$\gxy$  & s \\
	\hline	\hline
$\mathcal{Q}_4^1$ &	$2^4:\alt{5}$	&	$2^{2+2+2}:\sym{3}$	&	$2^4:\alt{4}$ & 4 \\	\hline
$\mathcal{Q}_4^2$ &	$2^4:\alt{5}$	&	$2^{2+2+2}:\cyc{6}$	&	$2^4:\alt{4}$ & 4 \\	\hline
$\mathcal{Q}_4^3$ &	$2^4:(\alt{5}\times \cyc{3})$	&	$(2^{2+4}:3) : \sym{3}$	&	$2^{2+4}:3^2$ & 4 \\	\hline
$\mathcal{Q}_4^4$ &	$2^4:(\alt{5}\times \cyc{3})$	&	$(2^{2+4}:3) : \cyc{6}$	&	$2^{2+4}:3^2$ & 4 \\ 	\hline
$\mathcal{Q}_4^5$ & 	$2^4:\sym{5}$	&	$2^{2+4+1}:\sym{3}$	&	$2^{2+4}:\sym{3}$ & 4 \\ 	\hline
$\mathcal{Q}_4^6$ & 	$2^4 : \sym{5} \foo{scale=.15} \sym{3}$	&	$2^{2+4}:\sym{3}^2$	&	$2^{2+4}:\sym{4}\foo{scale=.15} \sym{3}$ & 4\\ 	\hline
$\mathcal{Q}_5^1$ &  	$2^6: \sym{5}\foo{scale=.15} \sym{3}$	&	$(2^6 : (\alt{4} \times \cyc{3})) : \cyc{4}$	&	$2^6 : \sym{4}\foo{scale=.15} \sym{3}$ & 5 

\end{tabular}
\caption{Amalgams with $\vst{xy}{1}\neq 1$.}
\label{table:weisstable}
\end{center}
\end{table}

\begin{rem}
In Table \ref{table:weisstable} we have given a description of the groups in terms of a factors appearing in a normal series, but this  does not determine the group nor the amalgam uniquely. As we have mentioned, the amalgams in rows 1-6 have completions inside $\aut{\PSL{3}{4}}$, we now give explicit constructions. Let $L=\PSL{3}{4}$ and identify $L$ with a subgroup of $A=\aut{L}$ (see \cite[pg.23]{atlas} for various properties of $L$ and $A$). We can choose parabolic subgroups $P_1$ and $P_2$ intersecting in a Borel subgroup $B$ such that $P_1$ and $B$ are normalised by the outer automorphisms $f$ and $p$ (which generate a subgroup isomorphic to $\sym{3}$ in $A$) and $P_1$ and $P_2$ are permuted by the ``graph" automorphism $g$ (so that $\grp{f,p,g}\cong \sym{3}\times \cyc{2}$ is a complement to $L$ in $A$).

For $i\in [1,6]$ we will define subgroups $G_i$ of $A$ which are completions for the amalgams $\mathcal{Q}_4^i$. We build amalgams $\mathcal{A}_i=(A_1,A_2,A_{12})$ over the subgroups $P_1$ and $B$ so that $\mathcal{A}_i$ has the same type as $\mathcal{Q}_4^i$. Then since $\mathcal{Q}_4^i$ is the unique amalgam of that type, we see that $G_i=\grp{A_1,A_2}$ is a completion of $\mathcal{Q}_4^i$. Beginning with the last two, let $\mathcal{A}_5=(\grp{P_1,f},\grp{B,f,g},\grp{B,f})$ and let $\mathcal{A}_6=(\grp{P_1,f,p},\grp{B,f,p,g},\grp{B,f,p})$.  Then set $\mathcal{A}_{1}=(P_1,\grp{B,fg},B)$ and $\mathcal{A}_{2}=(P_1,\grp{B,g},B)$. In a similar fashion, we obtain $\mathcal{A}_{3}=(\grp{P_1,p},\grp{B,p,fg},\grp{B,p})$ and $\mathcal{A}_{4}=(\grp{P_1,p},\grp{B,p,g},\grp{B,p})$.

The group $K=\aut{\mathrm{Sp}_4(4)}$ is a completion of $\mathcal{Q}_5^1$. There is an element of order 4, $f$, which generates a complement to $J=\mathrm{Inn}(\mathrm{Sp}_4(4))$ (see \cite[pg.44]{atlas}). As above, we can take parabolic subgroups of $J$, $P_1$ and $P_2$ say, which come from different classes and intersect in a Borel subgroup $B$ such that $f$ interchanges $P_1$ and $P_2$ and  $f^2$ normalises both. Set $A_1=\grp{P_1,f^2}$, $A_2=\grp{B,f}$ and $A_{12}=\grp{B,f^2}$. The amalgam $\mathcal{A}=(A_1,A_2,A_{12})$ is of the same type as $\mathcal{Q}_5^1$. Since $\mathcal{Q}_5^1$ is the unique amalgam of this type, $\aut{\mathrm{Sp}_4(4)}$ is a completion of $\mathcal{Q}_5^1$. Note that $A_1/\core{A_1}{A_{12}}\cong \sym{5}$, so the local action at a vertex is the full symmetric group. On the other hand, there is no index 2 subgroup of $K$ which contains $f$. In some sense, this can be seen as the reason that there  is no $5$-transitive symmetric quintic graph in which the local action is the alternating group of degree 5.
\end{rem}

\section{The soluble case when $s\leq 3$}
\label{sec:gxy1=1sol}

In this section we assume that $s\leq 3$ and $\gx$ is soluble. This situation was investigated in \cite{zhoufeng} where the isomorphism type of $\gx$ and value of $s$ is determined.

\begin{thm}
Suppose that $\gx$ is soluble and $s\geq 1$. Then $s\leq 3$ and $\gx$ is isomorphic to one of $\cyc{5}$, $\dih{10}$, $\dih{20}$ if $s=1$, one of $\frtw$, $\frtw \times \cyc{2}$ if $s=2$ or $\frtw\times \cyc{4}$ if $s=3$.
\begin{proof}
See \cite[Theorem 4.1]{zhoufeng}.
\end{proof}
\end{thm}

We now determine the isomorphism type of the group $\ga$ and the embedding $\gxy \rightarrow \ga$. First suppose that $\vstx{1}=1$. This gives us the list of seven amalgams in Table \ref{table:solgx1}. To find the list, we use the fact that $\gxy$ is uniquely determined by $\gx$, and  we consider each of the groups of order $2|\gxy|$ which has a subgroup isomorphic to $\gxy$. This gives the list in Table \ref{table:solgx1}.

\begin{table}[h]
\begin{center}\begin{tabular}{ c| c| c| c |c}
\hline	\hline
Amalgam	& $\gx$	& 	$\ga$	&	$\gxy$  & s \\
 	\hline	\hline
$\mathcal{Q}_1^1$ & 	$\cyc{5}$	&	$\cyc{2}$	&	1 & 1 \\	\hline
$\mathcal{Q}_1^2$ & 	$\dih{10}$	&	$2^2$	&	$\cyc{2}$ & 1 \\	\hline
$\mathcal{Q}_1^3$ &  $\dih{10}$	&	$\cyc{4}$	&	$\cyc{2}$ & 1 \\	\hline
$\mathcal{Q}_2^1$ & 	$\frtw$	&	$\cyc{4}\times\cyc{2}$	&	$\cyc{4}$ & 2 \\ 	\hline
$\mathcal{Q}_2^2$ &  	$\frtw$	&	$\cyc{8}$	&	$\cyc{4}$ & 2 \\ 	\hline
$\mathcal{Q}_2^3$ &  	$\frtw$	&	$\dih{8}$	&	$\cyc{4}$ & 2\\ 	\hline
$\mathcal{Q}_2^4$ &  	$\frtw$	&	$\mathrm{Q}_8$	&	$\cyc{4}$ & 2 
\end{tabular}
\caption{Amalgams with soluble vertex stabilisers and $\vstx{1}=1$.}
\label{table:solgx1}\end{center}
\end{table}

From now on we assume that $\vstx{1}\neq 1$. Then $\vstx{1}$ is isomorphic to its projection over $\vsty{1}$ since $\vst{xy}{1}=1$. Furthermore, $[\vstx{1},\vsty{1}]=1$, so $G_{xy}$ contains a normal subgroup isomorphic to $\vstx{1}\times \vstx{1}$. 

\begin{lem}
Suppose that  $\gx \cong \dih{20}$. Then $\ga \cong \dih{8}$.
\begin{proof}
As $\gx \cong\dih{20}$ we see $\vstx{1}$ has order 2 and $\gxy \cong 2^2$. Then $\ga$ is a non-abelian group of order 8 with an elementary abelian subgroup of order 4. It follows that $\ga \cong \dih{8}$.
\end{proof}
\end{lem}

In the next lemma we find the edge stabilisers have order 16. Recall the modular group, $\mathrm{M}_{16}$ of order 16 has presentation $\grp{u,v\mid u^8=1,v^2=1,u^v=u^5}$, and as a subgroup of $\sym{8}$ is generated by the permutations $(1,2,3,4,5,6,7,8)$ and $(2,6)(4,8)$. By $\mathrm{N}_{16}$ we denote the group $\grp{(1,2,3,4)(5,6,7,8),(5,7)(6,8),(1,5)(2,6)(3,7)(4,8)}$. Observe that $\mathrm{N}_{16}$ has a central cyclic subgroup of order 4, modulo which it is elementary abelian of order 4.

\begin{lem}
\label{lem:edge16}
Suppose that $\gx \cong \frtw \times \cyc{2}$. Then $\ga \cong \mathrm{M}_{16}$ or $\ga\cong \mathrm{N}_{16}$.
\begin{proof}
We have $\gxy \cong 4 \times 2$, fix  notation $\gxy=\grp{h,j}$ where $h$ has order 4 and $j$ has order 2. Additionally, we may assume that $\grp{j}=\vstx{1}$ and $\grp{h^2j}=\vsty{1}$ since $j$ is not a square in $\gxy$. We know there is $t\in\ga$ such that $j^t=h^2j$, and we choose such a $t$ with order as small as possible. If $t$ has order 2, then we find that $\ga\cong N_{16}$, otherwise $t$ has order 4 or 8. If $t$ has order 8, then (after changing notation if necessary) we have $t^2=h$ and so $j^t=h^2j=t^4j$ implies that $t^j=t^5$ and we see $\ga\cong M_{16}$. It remains to see that $t$ cannot have order 4. 

There are exactly two cyclic subgroups of order 4 in $\gxy$ and these are generated by $h$ and $hj$ respectively. We claim that $t$ centralises one of these subgroups. First, assume that $h^t=hj$ or $h^t=h^3j$. Then $t^2\in\gxy$, so $h^{t^2}=h$. On the other hand, both of $h^t=hj$ and $h^t=h^3j$ imply that $h^{t^2}=h^3$, whence $h=h^3$, a contradiction. Hence either $h^t=h$, in which case $t$ centralises $\grp{h}$ or $h^t=h^3$. Then we find that $(hj)^t=h^3h^2j=hj$, so $t$ centralises $\grp{hj}$. In both cases, we find an element of order 4, $k$ say, in $\gxy$ which is centralised by $t$. Hence $t^2=k^2$ and so $(tk)^2=1$, but $tk\notin\gxy$, and this contradicts our choice of $t$ with minimal order.
\end{proof}
\end{lem}

\begin{lem}
Suppose that $\gx \cong \frtw \times \cyc{4}$. Then $\ga \cong \cyc{4}\wr\cyc{2}$.
\begin{proof}
Since $\vstx{1}\cong \cyc{4}$, we have $\gxy=\vstx{1}\vsty{1} \cong \cyc{4} \times \cyc{4}$. Choose $q$ of least order such that $q \notin \gxy$, we claim $q$ has order 2. Writing $\vstx{1}=\grp{a}$, set $b=a^q$, then $\vsty{1}=\grp{b}$ and $(a^i)^q=b^i$ for $i\in \field{N}$. Since $\ga$ is non-abelian, it follows that $\zent{\ga} = \grp{ab}$. Now $q^2 \in \gxy$ which is abelian, so $q^2 \in \zent{\ga}$. If $q^2=1$ we are done. Suppose first that $q^2=a^2b^2$. Then $(qab)^2=1$, and $qab \notin \gxy$ since $q\notin \gxy$, this contradicts our choice of $q$. Similarly, if $q^2=ab$ or $q^2=a^3b^3$,  we find that $qb^3$, respectively, $qb$, are involutions, and do not lie in $\gxy$. Thus we may assume $q$ is an involution, and therefore $\ga \cong \cyc{4} \wr \cyc{2}$.
\end{proof}
\end{lem}

This completes the identifications of the vertex and edge stabilisers when both of these groups are soluble. The full list is in Table \ref{table:solgxy1}.

\begin{table}[ht]
\begin{center}
\begin{tabular}{ c|  c| c| c |c}
\hline	\hline
Amalgam &	$\gx$	& 	$\ga$	&	$\gxy$ & s\\
	\hline	\hline
$\mathcal{Q}_1^4$ &	$\dih{20}$	&	$\dih{8}$	&	$2^2$ & 1\\	\hline
$\mathcal{Q}_2^5$ &	$\frtw \times \cyc{2}$	&	$\mathrm{N}_{16}$	&	$\cyc{4}\times \cyc{2}$ & 2\\ 	\hline
$\mathcal{Q}_2^6$ & 	$\frtw \times \cyc{2}$	&	$\mathrm{M}_{16}$	&	$\cyc{4}\times \cyc{2}$ & 2\\ 	\hline
$\mathcal{Q}_3^1$ & 	$\frtw \times \cyc{4}$	&	$\cyc{4}\wr\cyc{2}$	&	$\cyc{4}\times \cyc{4}$ & 3 

\end{tabular}
\caption{Amalgams with soluble vertex stabilisers, $\vst{xy}{1}=1$ and $\vstx{1}\neq 1$.}
\label{table:solgxy1}
\end{center}
\end{table}

\section{The non-soluble case when $s\leq 3$}
\label{sec:gxy1=1nonsol}

Throughout this section we assume that $s\leq 3$ and  $\gx/\vstx{1}$ is non-soluble. Proposition \ref{prop:solofgx} and  Lemma \ref{lem:subgpsofs5} show that $\gx/\vstx{1}$ contains a normal subgroup isomorphic to $\alt{5}$. When $\vstx{1}=1$ the existence in $\ga$ of a normal subgroup isomorphic to $\alt{4}$ or $\sym{4}$ readily implies that $\ga$ is one of the groups in column 2 of Table \ref{table:nonsolgx1=1}.

\begin{table}[ht]
\begin{center}
\begin{tabular}{c |  c| c| c |c}
\hline	\hline
Amalgam &	$\gx$	& 	$\ga$	&	$\gxy$ & s \\
	\hline	\hline
$\mathcal{Q}_2^7$ &  	$\alt{5}$	&	$\sym{4}$	&	$\alt{4}$ & 2 \\ 	\hline
$\mathcal{Q}_2^8$ &  	$\alt{5}$	&	$\alt{4}\times \cyc{2}$	&	$\alt{4}$ & 2 \\ 	\hline
$\mathcal{Q}_2^9$ &  	$\sym{5}$	&	$\sym{4}\times \cyc{2}$	&	$\sym{4}$ & 2
\end{tabular}
\caption{Amalgams with non-soluble vertex stabilisers and $\vstx{1}=1$.}
\label{table:nonsolgx1=1}
\end{center}
\end{table}

From now on we  assume that $\vstx{1}\neq 1$. We have the following easy consequence.

\begin{prop}
\label{prop:isomorphismsofgx1}
There are isomorphisms $\vstx{1}\vsty{1}\cong\vstx{1}\times \vsty{1}$ and $\vstx{1}\cong\vstx{1}\vsty{1}/\vsty{1}$. The latter subgroup is a normal subgroup of $\gxy/\vsty{1}$.
\begin{proof}
The first isomorphism follows directly from Lemma \ref{lem:gxy1=1} and the normality of both $\vstx{1}$ and $\vsty{1}$ in $\gxy$. The second isomorphism follows from Lemma \ref{lem:gxy1=1} and an isomorphism theory. The second part follows since $\vst{x}{1}\vsty{1}$ is normal in $\ga$.
\end{proof}
\end{prop}

\begin{lem}
\label{lem:cent of gx1gy1}
We have $\cent{\ga}{\vstx{1}\vsty{1}} =\cent{\gxy}{\vstx{1}\vsty{1}}=\cent{\gx}{\vstx{1}\vsty{1}}=\zent{\vstx{1}\vsty{1}}$. 
\begin{proof}
Set $C_e=\cent{\ga}{\vstx{1}\vsty{1}}$ and $C_x=\cent{\gx}{\vstx{1}\vsty{1}}$. The first equality will follow once we have shown $C \leq \gxy$. If this were not the case, then $C$ contains an element permuting $x$ and $y$ non-trivially. Also we see that $[C,\vstx{1}]\leq[C,\vstx{1}\vsty{1}]=1$, hence $\vstx{1}$ is a normal subgroup of $\grp{\gx,C}$, which acts transitively on $\gam$. Then Proposition \ref{prop:fund prop of gam} (ii) forces $\vstx{1}=1$, a contradiction. Now $\zent{\vstx{1}\vsty{1}}\leq C_e \leq C_x$, so it remains to see that the latter subgroup is contained in $ \vstx{1}\vsty{1}$. Using Proposition \ref{prop:isomorphismsofgx1} and the isomorphisms $\gx/\vstx{1}\cong\alt{5}$ or $\gx/\vstx{1}\cong\sym{5}$, we see that normal subgroups of $\gxy/\vstx{1}$ contain their centralisers in $\gx/\vstx{1}$, therefore 
\[ \begin{split}
C_x\vstx{1}/\vstx{1} \leq \cent{\gx/\vstx{1}}{\vstx{1}\vsty{1}/\vstx{1}}\leq \vstx{1}\vsty{1}/\vstx{1}
\end{split} \]
and so $C_x \leq C_x\vstx{1} \leq \vstx{1}\vsty{1}$ as required.
\end{proof}
\end{lem}

\begin{lem}
\label{lem:ga acts faithfully}
 The group $\vstx{1}$ is isomorphic to either $\alt{4}$ or $\sym{4}$. Moreover, $\ga$ acts faithfully on $\vstx{1}\vsty{1}$ by conjugation.
\begin{proof}
Proposition \ref{prop:isomorphismsofgx1} shows that $\vstx{1}$ is isomorphic to one of $2^2$, $\alt{4}$ and $\sym{4}$. Let us assume $\vstx{1}\cong 2^2$. Then Lemma \ref{lem:cent of gx1gy1} gives $\vstx{1}\vsty{1} = \cent{\gx}{\vstx{1}\vsty{1}}\leq\cent{\gx}{\vstx{1}}$. Since $\vstx{1}<\cent{\gx}{\vstx{1}}\nml \gx$ we find that elements of order 3 in $\gx$ centralise $\vstx{1}$. In particular, all elements of order 3 in $\gxy$ centralise $\vstx{1}$. Passing to $\gxy/\vsty{1}$ we see that elements of order 3 here centralise $\vstx{1}\vsty{1}/\vsty{1}$, which produces a subgroup isomorphic to $2^2 \times 3$. But $\gxy/\vsty{1}$ contains no such subgroup, a contradiction.

We now have $\zent{\vstx{1}\vsty{1}}=1$, so Lemma \ref{lem:cent of gx1gy1} provides the final statement.
\end{proof}
\end{lem}

We define $A=\aut{\vstx{1}\vsty{1}}\cong\sym{4}\wr 2$. This isomorphism follows from the observation there are exactly two normal subgroups isomorphic to $\alt{4}$ in $\vstx{1}\vsty{1}$. Lemma \ref{lem:cent of gx1gy1} allows us to identify $\ga$ with a subgroup of $A$. Note that $\oP{2}{A}\cong\alt{4}\times\alt{4}$ and by Lemma \ref{lem:cent of gx1gy1} $\oP{2}{A} \leq \vstx{1}\vsty{1}$. Thus we see $\ga/\oP{2}{A}$ in the quotient $A/\oP{2}{A}\cong\dih{8}$. We use these observations below.

\begin{lem}
Suppose that $\gx/\vstx{1} \cong \alt{5}$. Then $\gx \cong \alt{5}\times \alt{4}$ and $\ga \cong \alt{4}\wr 2$.
\begin{proof}
Lemma \ref{lem:ga acts faithfully} gives $\vstx{1}\cong \alt{4}$. Since $\cent{\gx}{\vstx{1}}\cap \vstx{1}=1$, we see that $\cent{\gx}{\vstx{1}}$ is either trivial or isomorphic to $\alt{5}$. Since $\gx/\cent{\gx}{\vstx{1}}$ embeds into $\sym{4}$, we have $\gx=\cent{\gx}{\vstx{1}}\vstx{1} \cong \alt{5}\times\alt{4}$. Now $\gxy=\vstx{1}\vsty{1}\cong\alt{4}\times\alt{4}$, and so $\ga \cong \alt{4}\wr \cyc{2}$.
\end{proof}
\end{lem}

\begin{lem}
Suppose that $\gx/\vstx{1} \cong \sym{5}$ and $\vstx{1}\cong \sym{4}$. Then $\gx \cong \sym{5}\times \sym{4}$ and $\ga \cong \sym{4}\wr 2$.
\begin{proof}
Since $\vstx{1}\cap \cent{\gx}{\vstx{1}} =1$, and $\vstx{1}$ is isomorphic to $\aut{\vstx{1}}$, we have $\gx=\cent{\gx}{\vstx{1}}\vstx{1}$ and so $\cent{\gx}{\vstx{1}}\cong\sym{5}$. Now $\gxy=\vstx{1}\vsty{1} \cong \sym{4}\times\sym{4}$ and so $\ga\cong\aut{\gxy}$.
\end{proof}
\end{lem}

Finally, we have to deal with the possibility that $\gx/\vstx{1}\cong \sym{5}$ and $\vstx{1}\cong \alt{4}$. There are two types of amalgam which have this property. 

\begin{lem}
Suppose  $\gx/\vstx{1} \cong \sym{5}$ and $\vstx{1}\cong \alt{4}$. Then $\gx \cong \sym{5}\foo{scale=.15}\sym{4}$.
\begin{proof}
Let $C=\cent{\gx}{\vstx{1}}$. Then $C \cap \vstx{1} =1$, but $\vsty{1}\leq C$, so either $C\cong \alt{5}$ or $C\cong \sym{5}$ and $\gx=C\vstx{1}$ holds. Assuming $\gx=C\vstx{1}$ we find $\gxy=\vstx{1}\cent{\gxy}{\vstx{1}}$. But now $\vsty{1}\leq \cent{\gxy}{\vstx{1}}$, so $\gxy/\vsty{1}  \cong \alt{4} \times 2$, which is a contradiction to $\gy/\vsty{1}\cong\gx/\vstx{1}\cong\sym{5}$. Hence $C\vstx{1} \cong \alt{5} \times \alt{4}$ and has index 2 in $\gx$. Since $\gx/C$ embeds into $\aut{\vstx{1}}$ we have $\gx/C \cong \sym{4}$. Thus $\gx$ has a  subgroup of index 2 isomorphic to $\alt{5} \times \alt{4}$ and quotients isomorphic to $\sym{5}$ and $\sym{4}$. It follows that $\gx \cong \sym{5} \foo{scale=.15} \sym{4}$.
\end{proof}
\end{lem}

Here there are two different types of amalgams, corresponding to two different possibilities for $\ga$. These groups differ in the isomorphism type of $\ga/\vstx{1}\vsty{1}$, which has order 4, but is either cyclic or elementary abelian. 

\begin{lem}
\label{lem:defnofl1l2}
Suppose that $\gx/\vstx{1}\cong\sym{5}$ and $\vstx{1}\cong\alt{4}$. Then $\ga$ is isomorphic to one of 
\begin{itemize}
\item[] $\mathrm{L}_1 =\grp{(1,2,3),(2,3,4),(5,6,7),(6,7,8),(1,2)(5,6),(1,5)(2,6)(3,7)(4,8)}$,
\item[] $\mathrm{L}_2=\grp{(1,2,3),(2,3,4),(5,6,7),(6,7,8),(1,6,2,5)(3,7)(4,8)}$.
\end{itemize}
\begin{proof}
Comparing orders, we see that $|\ga:\vstx{1}\vsty{1}|=4$. By Lemma \ref{lem:ga acts faithfully} and the remarks following, $\ga$ can be identified with a subgroup of index two in $A$ which contains the characteristic subgroup $\oP{2}{A}$. There are precisely three of these, $\mathrm{L}_1$ and $\mathrm{L}_2$ above and $\mathrm{L}_3 \cong \sym{4}\times\sym{4}$. Identifying $\vstx{1}$ with its image in $A$ we see $\vstx{1}\nml \mathrm{L}_3$, so we must have $\ga \cong \mathrm{L}_1$ or $\ga\cong \mathrm{L}_2$.
\end{proof}
\end{lem}

The final list that we have compiled in this section is in Table \ref{table:nonsolgxy1=1}.

\begin{table}[ht]
\begin{center}
\begin{tabular}{c|  c| c| c | c}
\hline \hline
Amalgam & 	$\gx$	& 	$\ga$	&	$\gxy$ & s\\
	\hline	\hline
$\mathcal{Q}_3^2$ &  	$\alt{5} \times \alt{4}$	&	$\alt{4}\wr \cyc{2}$	&	$\alt{4}\times\alt{4}$ & 3 \\ 	\hline
$\mathcal{Q}_3^3$ & 	$\sym{5}\foo{scale=.15}\sym{4}$	&	$\mathrm{L}_1$	&	$\sym{4}\foo{scale=.15}\sym{4}$ & 3\\ 	\hline
$\mathcal{Q}_3^4$ &	 	$\sym{5}\foo{scale=.15}\sym{4}$	&	$\mathrm{L}_2$	&	$\sym{4}\foo{scale=.15}\sym{4}$ & 3 \\ 	\hline
$\mathcal{Q}_3^5$ & 	$\sym{5}\times\sym{4}$	&	$\sym{4}\wr \cyc{2}$	&	$\sym{4}\times\sym{4}$ & 3 
\end{tabular}
\caption{Amalgams with non-soluble vertex stabilisers, $\vst{xy}{1}=1$ and $\vstx{1}\neq 1$.}
\label{table:nonsolgxy1=1}
\end{center}
\end{table}

Note that so far, even though we have given the amalgams a name, we have not determined how many amalgams each type determines. This problem is addressed in the next section.

\section{Uniqueness and Presentations}
\label{sec:uniandpres}

Using Goldschmidt's Lemma we verify that each of the primitive amalgams we have found is unique. We begin with the amalgams found last, those in Table \ref{table:nonsolgxy1=1}.

\begin{lem}
Let $\mathcal{A}=(\gx,\ga,\gxy)$ be an amalgam from Table \ref{table:nonsolgxy1=1}. Then $\mathcal{A}$ is the unique amalgam of this type.
\begin{proof}
We leave the reader to verify that $\aut{\gx}\cong \sym{5}\times\sym{4}$ and $\aut{\ga}\cong\sym{4}\wr\cyc{2}$ for each of the amalgams. Then the image of $\norm{\aut{\gx}}{\gxy}$ in $\aut{\gxy}$ is the subgroup isomorphic to $\sym{4}\times\sym{4}$. Since there is an inner automorphism of $\ga$ which normalises $\gxy$ and swaps the factors, the image of this element in $\aut{\gxy}$ lies outside the $\sym{4}\times\sym{4}$ subgroup. Hence $\aut{\gxy}=A_1^*A_2^*$, so there is a unique amalgam of these types by the Goldschmidt Lemma.
\end{proof}
\end{lem}

The next cases are easier still.

\begin{lem}
Let $\mathcal{A}=(\gx,\ga,\gxy)$ be an amalgam from Table \ref{table:nonsolgx1=1}. Then $\mathcal{A}$ is the unique amalgam of this type.
\begin{proof}
Since $\aut{\gxy}\cong\sym{4}$ for each of these amalgams, $\cent{\aut{\gx}}{\gxy}=1$ and $\norm{\aut{\gx}}{\gxy}\cong\sym{4}$, we find $\aut{\gxy}=A_1^*$, so there is a unique amalgam.
\end{proof}
\end{lem}

In the next two lemmas we need to add the word ``primitive" to make a statement about uniqueness.

\begin{lem}
Let $\mathcal{A}=(\gx,\ga,\gxy)$ be an amalgam from row 1 of Table \ref{table:solgxy1}. There are two isomorphism classes of amalgams of this type, precisely one is primitive.
\begin{proof}
We see that $\aut{\gxy}\cong\sym{3}$. After choosing a labelling, one finds that $A_1^*$ is the subgroup $\grp{(1,2)}$ and that $A_2^*=\grp{(2,3)}$. Hence there are two $(A_1^*,A_2^*)$ double cosets in $\aut{\gxy}$. For both of these amalgams we have $\zent{\gx}\zent{\ga}\leq \gxy$, but the primitive amalgam has $\zent{\gx} \cap \zent{\ga}=1$, and the non-primitive amalgam has $\zent{\gx}=\zent{\ga}$.
\end{proof}
\end{lem}

\begin{lem}
\label{lem:q31 unique}
Let $\mathcal{A}=(\gx,\ga,\gxy)$ be an amalgam from row 4 of Table \ref{table:solgxy1}. There are three isomorphism classes of amalgams of type $\mathcal{A}$ and precisely one is primitive.
\begin{proof}
We identify $\aut{\gxy}$ with the group $\GL{2}{\mathbb{Z}/4\mathbb{Z}}$. Using generators for the group $\aut{\gx}\cong\frtw\times\dih{8}$ and $\aut{\ga}\cong\dih{8}:2^2$, we find that $A_1^*\cong\dih{8}$ and $A_2^*\cong 2^3$ and are generated by the matrices
\[ \begin{split}
A_1^*=\left \langle \left [ \begin{array}{cc} 1 & 0 \\ 0 & 3 \end{array} \right ], \left [ \begin{array}{cc} 1 & 1 \\ 0 & 1 \end{array} \right ] \right \rangle,\ A_2^*=\left \langle \left [ \begin{array}{cc} 3 & 0 \\ 0 & 3 \end{array} \right ], \left [ \begin{array}{cc} 0 & 1 \\ 1 & 0 \end{array} \right ] , \left [ \begin{array}{cc} 2 & 1 \\ 1 & 2 \end{array} \right ]  \right \rangle.
\end{split}
\]
Either by hand or with the aid of {\sc Magma} or {\sc Gap} one can verify that there are three $(A_1^*,A_2^*)$ double cosets in $\aut{\gxy}$, and so there are  three isomorphism classes of amalgams with this type. In Example \ref{eg:example1} we constructed three pairwise non-isomorphic amalgams of this type and precisely one is primitive.
\end{proof}
\end{lem}

\begin{lem}
There is a unique class of amalgams of type $\mathcal{A}=(\gx,\ga,\gxy)$ where $\mathcal{A}$ comes from either row 2 or row 3 of Table \ref{table:solgxy1}.
\begin{proof}
We write $\gxy=\grp{x,y}$ where $x$ has order 4 and $y$ has order 2, and consider the action of the groups $\aut{\gxy}$, $\norm{\aut{\gx}}{\gxy}$ and $\norm{\aut{\ga}}{\gxy}$ on $\Omega=\{x,x^{-1},xy,x^{-1}y\}$, the elements of order 4 in $\gxy$. Since $\aut{\gxy}\cong\dih{8}$ acts faithfully on $\Omega$ we may write elements of $A_1^*$ and $A_2^*$ as permutations of $\{1,2,3,4\}$  (acting on subscripts after  labelling $x_1=x$, $x_2=x^{-1}$, $x_3=xy$, $x_4=x^{-1}y$). In both cases we see $A_1^*=\grp{(1,3)(2,4)}$ and $A_2^*$ contains the subgroup $\grp{(1,2)(3,4),(3,4)}$. Hence $\aut{\gxy}=A_1^*A_2^*$, so by the Goldschmidt Lemma there is a unique class of amalgams.

\end{proof}
\end{lem}

\begin{lem}
Suppose that $\mathcal{A}$ is an amalgam from  Table \ref{table:solgx1}. Then $\mathcal{A}$ is the unique  amalgam of this type.
\begin{proof}
For the amalgams in rows 1-3 there is nothing to prove since $\aut{\gxy}=1$. For the remaining amalgams $\aut{\gxy}\cong \cyc{2}$. Inspecting $\aut{\ga}$ we find an element which inverts $\gxy$ in all cases, so we are done.
\end{proof}
\end{lem}

The final result of this paper is a presentation for the universal completion of each of the primitive amalgams which are original in this paper. For the remainder, that is $\mathcal{Q}_4^1$ - $\mathcal{Q}_4^6$ and $\mathcal{Q}_5^1$, we quote the presentations from \cite{Weiss-pres} (adjusted so that the commutators fit with our notation).

\begin{thm}
\label{thm:presses}
Suppose that $\mathcal{Q}_i^j=(\gx,\ga,\gxy)$ is a primitive amalgam of degree (5,2) and that $G_i^j$ is the universal completion. Then a presentation for $G_i^j$ is given in Tables \ref{table:pres} and \ref{table:pres2}.
\begin{proof}
We take a set of generators $X$ for $\gxy$ with relations $R$. Then $G_i^j$ has a presentation $\grp{X,a,b \mid R,S,T}$ where $a$, $b$, $S$ and $T$ are such that $\grp{X,a \mid R, S }=\gx$ and $\grp{X,b \mid R,T}=\ga$. Note that two extra generators suffices since $\gxy$ is a maximal subgroup of both $\gx$ and $\ga$. When $\ga$ does not split over $\gxy$ we may adjust $R$ so that we have $b^2 \in X$, if this offers some advantage we do this. Of course there is no relation between $a$ and $b$. For $i \geq 4$ we have reproduced the presentations from \cite{Weiss-pres} in Table \ref{table:pres2}.
\end{proof}
\end{thm}

\begin{rem}
The presentations have been chosen for ease of use, certainly more efficient presentations exist. Commutators are written $[x,y]=x^{-1}y^{-1}xy$ and conjugation  $x^y=y^{-1}xy$.
\end{rem}

The universal completions for the amalgams $\mathcal{Q}_4^1$-$\mathcal{Q}_4^6$ are generated by elements $a$, $e_0$, $c$, $f$ and $g$. For $i\in \field{Z}$ we define $e_i:= a^i e_0 a^{-i}$ and $t=e_0e_3e_0$.
The universal completion of the amalgam $\mathcal{Q}_5^1$ is generated by elements $a$, $e_0$ and $c$, and as before set $e_i:= a^i e_0 a^{-i}$. We also define $t:=e_0e_4e_0$, $f:=a ca\inv $ and $g=(ta)^2$.

\begin{table}[ht!]
\setlength{\extrarowheight}{1.8pt}
\begin{center}
\begin{tabular}{ c |p{2cm}|p{10cm} }  \hline

Type & Generators	& Relations \\ \hline \hline
$\mathcal{Q}_1^1$ & $a,$ $b$ & $a^5$, $b^2$ \\ \hline
$\mathcal{Q}_1^2$ & $a$, $b$, $c$ & $a^5$, $b^2$,  $c^2$, $(ac)^2$, $(bc)^2$ \\ \hline
$\mathcal{Q}_1^3$ & $a$, $b$  & $a^5$, $ b^4$, $ (b^2 a)^2$ \\ \hline
$\mathcal{Q}_1^4$ &  $a$, $b$, $c$ &  $a^{5}$, $b^4$, $c^2$, $(b c)^2$,  $(ab^2)^2$, $[a,c]$ \\ \hline \hline
$\mathcal{Q}_2^1$ & $ a$, $b$, $c$ & $a^5$, $b^2$, $c^4$, $a^c a^3$, $[b, c]$ \\ \hline
$\mathcal{Q}_2^2$ & $a$, $b$ & $a^5$, $b^8$, $a^{b^2} a^3$  \\ \hline
$\mathcal{Q}_2^3$ & $a$, $b$, $c$ &  $a^5$, $b^2$, $c^4$, $ a^c a^3$, $(c b)^2$ \\ \hline
$\mathcal{Q}_2^4$ & $a$, $b$, $c$ & $ a^5$, $b^4$, $c^4$ $a^c  a^3$, $c^b c$  \\ \hline

$\mathcal{Q}_2^5$	&	$a$, $b$, $c$, $d$	&	$a^5$, $b^2$, $c^4$, $d^2$, $a^ca^3$, $[a,d]$, $[b,c]$, $[c,d]$, $d^bc^2d$ \\ \hline
$\mathcal{Q}_2^6$	&	$a$, $b$, $c$ & $a^5$, $b^8$, $c^2$, $ a^{b^2} a^3$, $b^c b^3$, $[a,c]$ \\ \hline

$\mathcal{Q}_2^7$	&	$a$, $b$, $c$, $d$	& $a^3$, $b^2$, $c^3$, $d^3$, $(dc)^2$, $(da)^2$, $c^a c^2 d$, $(bc)^2$, $b^dbc$\\ \hline
$\mathcal{Q}_2^8$	&	$a$, $b$, $c$, $d$	& $a^3$, $b^2$, $c^3$, $d^3$, $(dc)^2$, $(da)^2$, $c^a c^2 d$, $[b,c]$, $[b,d]$\\ \hline
$\mathcal{Q}_2^9$	&	$a$, $b$, $c$, $d$	& $a^5$, $b^2$, $c^4$, $d^2$, $(cd)^3$, $[b,c]$, $[b,d]$, $a^3cad$ \\ \hline \hline

$\mathcal{Q}_3^1 $ & $a$, $b$, $c$ & $ a^5$, $ b^2$,  $ c^4$, $ a^c a^3$, $ [a, c^b]$, $ [c, c^b]$ \\ \hline

$\mathcal{Q}_3^2 $ &  $a$, $b$, $c$, $d$, $e$, $f$ & $a^3$, $ b^2 $, $ c^3$, $ d^3$,  $ e^3$, $ f^3$, $ (fe)^2$, $ [e,  c]$, $ [f,  c]$,  $ [e,  d]$, $ [f,  d]$, $ (dc)^2$,  $ [e,  a]$, $ [f,  a]$, $ (ad)^2$, $ c^a c^2 d $, $ e^b c$, $ f^b d$ \\ \hline

$\mathcal{Q}_3^3$ & $a$, $b$, $c$, $d$, $e$, $f$, $g$ & $c^3$, $ d^3$, $ e^2$, $ f^3$, $ g^3$, $ (gf)^2$,  $ [f, c]$, $[g, c]$,  $[f, d]$, $[g, d]$,   $(dc)^2$, $ (ef)^2$, $ (ec)^2$, $ e^g   f^2   e$, $ e^d   c^2   e$, $a^3$, $[f, a]$, $[g, a]$, $(ad)^2$, $ e   e^a$, $b^2$, $ f^2   c^b$, $ g^2   d^b$, $ (e   b)^2 $ \\  \hline

$\mathcal{Q}_3^4$ & $a$, $b$, $c$, $d$, $e$, $f$  &   $c^3$,     $d^3$,     $e^3$,     $f^3$,    $(d c)^2$,  $  [c, e]$, $  [d,  e]  $, $  [c,  f]  $,    $  [d,  f]  $,    $(f   e)^2 $,   $b^4 $,        $c^2   e^b $,    $d^2   f^b $,    $e c^b $,    $d^b e  f^2 $,    $a^3 $,    $ [c,  a] $,    $ [d, a] $,    $(af)^2 $,    $  [b^2,   a] $   \\ \hline

$\mathcal{Q}_3^5$	& $a$, $b$, $c$, $d$, $e$, $f$	 &  $c^4$, $d^2$, $e^4$, $f^2$, $(c d)^3$, $(e f)^3$, $[c,e]$, $[c,f]$, $[d,e]$, $[d,f]$, $a^5$, $a^3 c a d$, $[a,e]$, $[a,f]$, $b^2$, $c^b e^3$, $d^b f$ \\  \hline \hline
\end{tabular}
\caption{Presentations for the universal completions of finite, primitive (5,2) amalgams with $s\leq 3$.}
 \label{table:pres}
 \end{center}
\end{table}

\begin{table}[ht!]
\setlength{\extrarowheight}{1.8pt}
\begin{center}
\begin{tabular}{ c |p{2cm}|p{10.5cm} }  \hline

Type & Generators	& Relations \\ \hline \hline
$\mathcal{Q}_4^1$	& $a$, $e_0$, $c$ & $e_0^2$,$c^3$,$(e_0 e_3)^3,$  $t c t ^{-1} c,$ $(e_0 c)^3,$ $(c e_0 e_3)^5,$ $t a t ^{-1} a,$ $[e_0,e_1],$ $[e_0,c e_1 c ^{-1}],$ $[e_0,e_2] e_1,$ $[e_0,c e_2 c ^{-1}] c ^{-1} e_1 c,$ $a c a ^{-1} c$  \\ \hline

$\mathcal{Q}_4^2$	& $a$, $e_0$, $c$ & $e_0^2,$ $c^3,$ $(e_0 e_3)^3,$ $t c t^{-1} c,$ $(e_0 c)^3,$ $(c e_0 e_3)^5,$ $t a t^{-1} a,$ $[e_0,e_1],$ $[e_0,c e_1 c^{-1} ] ,$ $[e_0,e_2] e_1,$ $[e_0,c e_2 c^{-1}] c^{-1} e_1 c,$ $[a,c]$ \\ \hline
$\mathcal{Q}_4^3$	& $a$, $e_0$, $c$, $f$ & $e_0^2,$ $c^3,$ $f^3,$ $(e_0 e_3)^3,$ $t c t^{-1} c,$ $(e_0 c)^3,$ $(c e_0 e_3)^5,$ $t a t^{-1}  a,$ $[e_0,e_1],$ $[e_0,c e_1 c^{-1}],$ $[e_0,e_2] e_1,$ $[e_0,c e_2 c^{-1}]  c^{-1}  e_1 c,$ $[c,a],$ $[c,f],$ $[e,f],$ $a f  (cfa)^{-1} $ \\ \hline

$\mathcal{Q}_4^4$	& $a$, $e_0$, $c$, $f$ & $e_0^2,  $ $c^3,  $ $f^3,  $ $(e_0 e_3)^3,  $ $t c t^{-1}  c,  $ $(e_0 c)^3,  $ $(c e_0 e_3)^5,  $ $t a t^{-1}  a,  $ $[e_0,e_1],  $ $[e_0,c e_1 c^{-1}],  $ $[e_0,e_2] e_1,  $ $[e_0,c e_2 c^{-1}] c^{-1} e_1 c,  $ $a c a^{-1} c,  $ $[c,f],  $ $[e,f],  $ $a f a^{-1} f c^{-1}  $ \\ \hline

$\mathcal{Q}_4^5$	& $a$, $e_0$, $c$, $g$ & $e_0^2, $ $c^3, $ $g^2, $ $(e_0 e_3)^3, $ $t c t^{-1} c, $ $(e_0 c)^3, $ $(c e_0 e_3)^5, $ $t a t^{-1} a, $ $[e_0,e_1], $ $[e_0,c e_1 c^{-1}], $ $[e_0,e_2] e_1, $ $[e_0,c e_2 c^{-1}] c^{-1} e_1 c,$ $[a,c],$ $[e_0,g],$ $[a,g],$ $g c g c $\\ \hline

$\mathcal{Q}_4^6$	& $a$, $e_0$, $c$, $f$, $g$ & $e_0^2, $ $c^3, $ $g^2, $ $f^3, $ $(e_0 e_3)^3, $ $t c t^{-1} c, $ $(e_0 c)^3, $ $(c e_0 e_3)^5, $ $t a t^{-1} a, $ $[e_0,e_1], $ $[e_0,c e_1 c^{-1}], $ $[e_0,e_2] e_1, $ $[e_0,c e_2 c^{-1}] c^{-1} e_1 c, $ $[c,a], $ $[e_0,g], $ $[a,g], $ $g c g c, $ $g f g f, $ $[c,f], $ $[e,f], $ $a f (cfa)\inv $ \\ \hline \hline

$\mathcal{Q}_5^1$	& $a,$ $e_0$, $c$ & $c^3,$ $e_0^2,$ $(e_0 e_4)^3,$ $t c t \inv c,$ $g^2,$ $[e_0,g],$ $[a,g],$ $c^g c,$ $(e_0 c)^3,$ $[e_2,c],$ $(ce_0e_4)^5,$ $[c,f],$ $af (cfa)^{-1},$ $[e_0,e_1],$ $[e_0,e_2],$ $[e_0,e_3]e_2e_1$ \\ \hline

\end{tabular}
\caption{Presentations for the universal completions of finite, primitive (5,2) amalgams with $s\geq 4$.}
 \label{table:pres2}
 \end{center}
\end{table}

\section{Proofs of the main theorems}

Here we tie together the results of the previous sections to provide proofs for the main results of this paper. Theorem \ref{thm:bigthm2} is part of Theorem \ref{thm:presses}. As remarked in the introduction and explained in Section \ref{sec:pre}, Corollary \ref{bigthm1} follows immediately from \ref{thm:bigthm}. We give below the proof of \ref{thm:bigthm}.

\begin{proof}[Proof of Theorem \ref{thm:bigthm}]
Let $\mathcal{A}$ be a primitive amalgam of degree (5,2). The type of $\mathcal{A}$ is found in Section \ref{sec:valofs} if $s\geq 4$ and in Sections \ref{sec:gxy1=1sol} and \ref{sec:gxy1=1nonsol} if $s\leq 3$. We determine that $\mathcal{A}$ is the unique primitive amalgam of this type in Section \ref{sec:uniandpres}.
\end{proof}

Finally we consider the values of $s$ in Table \ref{table:bigtable}. 
\begin{proof}[Proof of Corollary \ref{cor:stranscor}]
Since $s\geq 2$ if and only if $G_x$ acts 2-transitively on the cosets of $G_{xy}$ in $G_x$, and for $s\geq 4$ the values are known, we simply have to argue that $s=2$ or $s=3$ in the various instances. Observe that $s=3$ implies $4=|G_{xyz}:G_{xyzw}|$ for any 3-path $(x,y,z,w)$. Using Theorem \ref{thm:bigthm} for the amalgams $\mathcal{Q}_2^1$ - $\mathcal{Q}_2^9$ we see this doesn't hold. For the amalgams $\mathcal{Q}_3^1$ - $\mathcal{Q}_3^5$, not only does the equality hold, but we see the projection $G_{xyz}\vst{z}{1}/\vst{z}{1}$ is cyclic of order 4, hence $s=3$ for these amalgams.
\end{proof}

\bibliographystyle{plain}
\bibliography{../biblio}

\end{document}